\documentclass{amsart}

\usepackage{graphicx}
\usepackage[latin1]{inputenc}

\newtheorem{thm}{Theorem}[section]
\newtheorem{cor}[thm]{Corollary}
\newtheorem{lem}[thm]{Lemma}
\newtheorem{prop}[thm]{Proposition}
\newtheorem{exa}[thm]{Example}
\newtheorem{dfn}[thm]{Definition}
\newtheorem{rem}[thm]{Remark}
\numberwithin{equation}{section}

\begin{document}

\title[]{On groups $G_{n}^{k}$, braids and Brunnian braids}

\author{S.Kim}

\address{Department of Fundamental Sciences, Bauman Moscow State Technical University, Moscow, Russia \\
ksj19891120@gmail.com}

\author{V.O.Manturov}

\address{Chelyabinsk State University and Bauman Moscow State Technical University, Moscow, Russia \\
vomanturov@yandex.ru}


\maketitle

\begin{abstract}
In \cite{Manturov} the second author defined the $k$-free braid group with $n$ strands $G_{n}^{k}$. These groups appear naturally as groups describing dynamical systems of $n$ particles in some ``general position''. Moreover, in \cite{ManturovNikonov} the second author and I.M.Nikonov showed that $G_{n}^{k}$ is closely related classical braids. The authors showed that there are homomorphisms from the pure braids group on $n$ strands to $G_{n}^{3}$ and $G_{n}^{4}$ and they defined homomorphisms from $G_{n}^{k}$ to the free product of $\mathbb{Z}_{2}$. That is, there are invariants for pure free braids by $G_{n}^{3}$ and $G_{n}^{4}$. 

On the other hand in \cite{FedoseevManturov} D.A.Fedoseev and the second author studied classical braids with addition structures: parity and points on each strands. The authors showed that the parity, which is an abstract structure, has geometric meaning -- points on strands. In \cite{Kim}, the first author studied $G_{n}^{2}$ with parity and points. the author construct a homomorphism from $G_{n+1}^{2}$ to the group $G_{n}^{2}$ with parity.

In the present paper, we investigate the groups $G_{n}^{3}$ and extract new powerful invariants of
classical braids from $G_{n}^{3}$. In particular, these invariants allow one to distinguish the non-triviality
of Brunnian braids. \\

\end{abstract}

\section{Introduction}

In \cite{Manturov}, the groups $G_{n}^{k}$ depending on two parameters $k,n \in \mathbb{N}, n>k$, were defined. The second author proved that many dynamical systems have topological invariants valued in $G_{n}^{k}$. It follows that there are homomorphisms from the $n$ strand pure braid group to $G_{n}^{3}$ and $G_{n}^{4} $ in \cite{ManturovNikonov}.
The groups $G_{n}^{k}$ are quite powerful by themselves; on the other hand, they admit various homomorphisms to free products of cyclic groups; this allows one to extract various easy-to-calculate invariants of classical braids valued in free groups (in \cite{ManturovNikonov} this approach was used to estimate the unknotting number for braids).

On the other hands, since the discovery of the parity by the second named author~\cite{Manturov1}, many invariants for classical links are extended to the case of virtual links. In~\cite{FedoseevManturov} D.A.Fedoseev and the second author applied the parity to the case of $n$ strand classical braids group. They showed that the parity for braids has a geometrical meaning -- the number of dots on strands. In~\cite{Kim}, the first author analogously enhanced those two structures to $G_{n}^{2}$. The author constructed homomorphism from $G_{n+1}^{2}$ to $G_{n}^{2}$ with parity. 

Brunnian braids are those $n$-strand braids which become trivial after removal of each strand. Such braids are closely related to various problems in geometry and topology, cryptography, etc, in particular, to the homotopy groups of spheres \cite{BardakovVershininWu}, \cite{BardakovMikhailovVershininWu1}, \cite{BardakovMikhailovVershininWu2}. The easy-to-calculate invariants from \cite{ManturovNikonov} $G_{n}^{3}$ valued in the free product of $\mathbb{Z}_{2}$ fail to recognize Brunnian braids(Lemma~\ref{failrecognizebrunnian}) when the number of strands is large.

The aim of the present paper is to construct more powerful invariants of braids (obtained by using homomorphisms from $G_{n}^{3}$ to other free products of cyclic groups and parity) and apply them to the recognize Brunnian braids. 

In section 2, we remind basic definitions and propositions. In section 3, we show that the image of a Brunnian braid in $PB_{n}$ by the homomorphism from $PB_{n}$ to $G_{n}^{3}$ defined in \cite{ManturovNikonov} is also Brunnian in $G_{n}^{3}$ (Theorem~\ref{thm_pres_brun}). In section 4 we show that the value of the invariant in \cite{ManturovNikonov} from $G_{n}^{3}$ for Brunnian braids is trivial and construct new invariant of classical braids from $G_{n}^{3}$ and $G_{n,p}^{2}$, which is a group obtained from $G_{n}^{2}$ by adding parity. And we show the example~\ref{recog_brun}, in which our new invariant distinguish the non-triviality of Brunnian braids. In section 5, we construct invariants from $G_{n,p}^{3}$, which is a group obtained from $G_{n}^{3}$ by adding parity.

\section{Basic definitions}

\begin{dfn}\cite{Manturov2}
Let $G_{n}^{2}$ be the group given by the presentation generated by $\{ a_{ij}~|~ \{i,j\} \subset \{1, \dots, n\}, i < j \}$ subject to the following relations:

\begin{enumerate}
\item $a_{ij}^{2} = 1$ for all $i \neq j$, 
\item $a_{ij}a_{kl} = a_{kl}a_{ij}$ for distinct $i,j,k,l$,
\item $a_{ij}a_{ik}a_{jk} = a_{jk}a_{ik}a_{ij}$ for distinct $i,j,k$.
\end{enumerate}
\end{dfn}

The group $G_{n}^{2}$ is also known as the pure free braid group on $n$ strands, see~\cite{Manturov}. Usually, free braids are depicted by free braid diagrams as follows.
\begin{dfn}
A {\em free braid diagram} is a graph, which is immersed in the rectangle $\mathbb{R} \times [0,1]$ with the following properties:
\begin{enumerate}
\item The graph vertices of valency $1$ are the points $[i,0]$ and $[i,1]$, $i = 1, \cdots, n$.
\item All other  vertices are $4$-valent vertices. 
\item For each $4$-valent vertex, edges are split into two pairs, two edges in which are called {\em opposite edeges}. The opposite edges in such vertices are at the angle of $\pi$.
\item The braid strands monotonously go down.
\end{enumerate}
A {\em braid strand} is an equivalence class of edges; two edges are called {\em equivalent} if there is a sequence of edges such that two consecutive edges in the sequence are opposite. 

For a free braid diagram, if $[i,0]$ and $[i,1]$ belong to the same strand for every $i \in \{1, \cdots, n\}$, then the diagram is called {\em a pure free braid diagram}.
\end{dfn}
In the group $G_{n}^{2}$, $a_{ij}$ corresponds to a crossing between $i$-th and $j$-th strands, see Fig.~\ref{exa-cro}. It is easy to see that the relations of $G_{n}^{2}$ correspond to Artin moves in Fig.~\ref{movesArtin}.
\begin{figure}[h!]
\begin{center}
 \includegraphics[width =6cm]{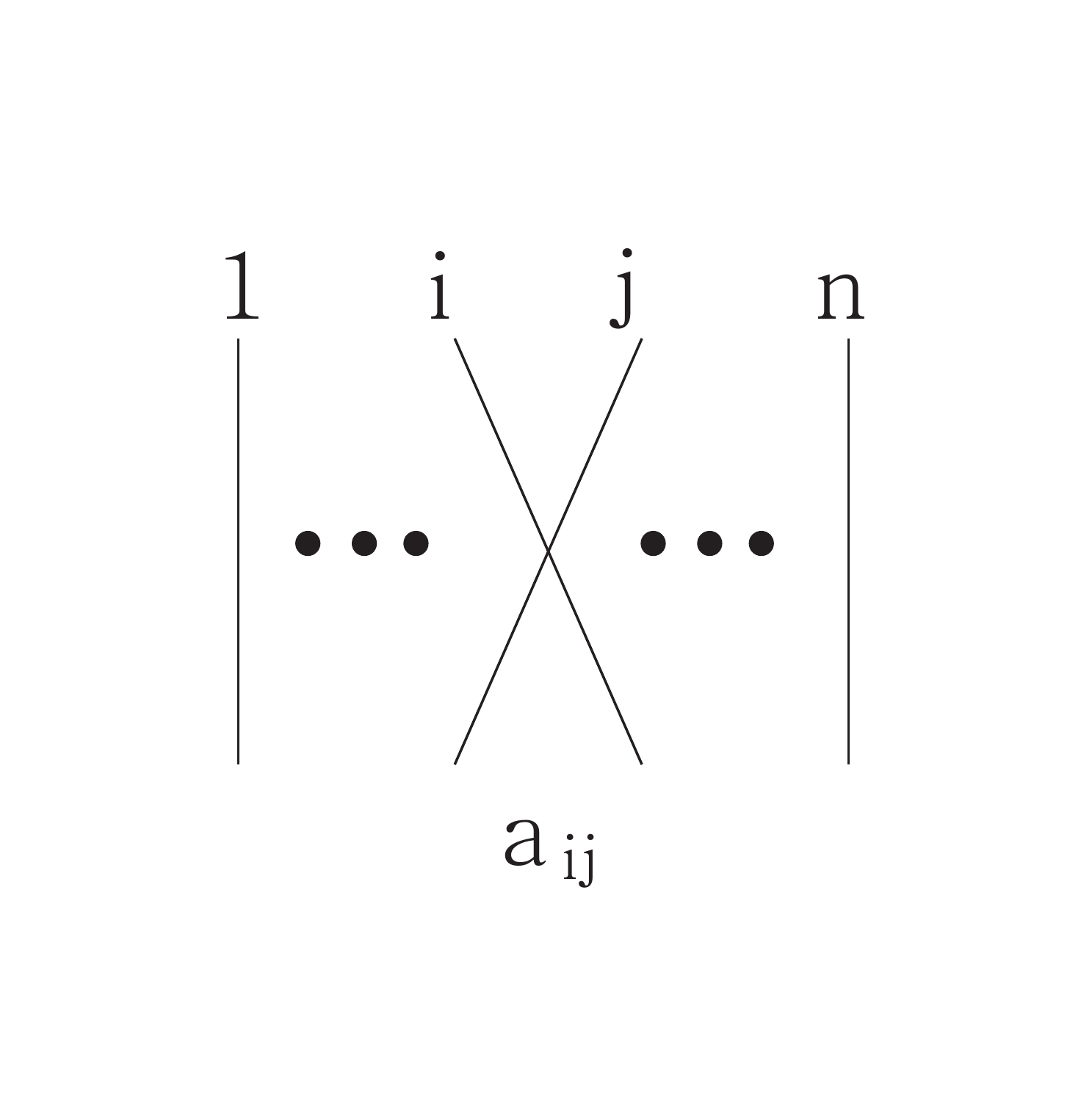}

\end{center}

 \caption{A crossing corresponding to $a_{ij}$}\label{exa-cro}
\end{figure}

\begin{dfn}
{\em A free braid on $n$ strand} is an equivalence class of free diagrams of $n$-strand braid under Artin moves for free braids in Fig.~\ref{movesArtin}.

\begin{figure}[h!]
\begin{center}
 \includegraphics[width =12cm]{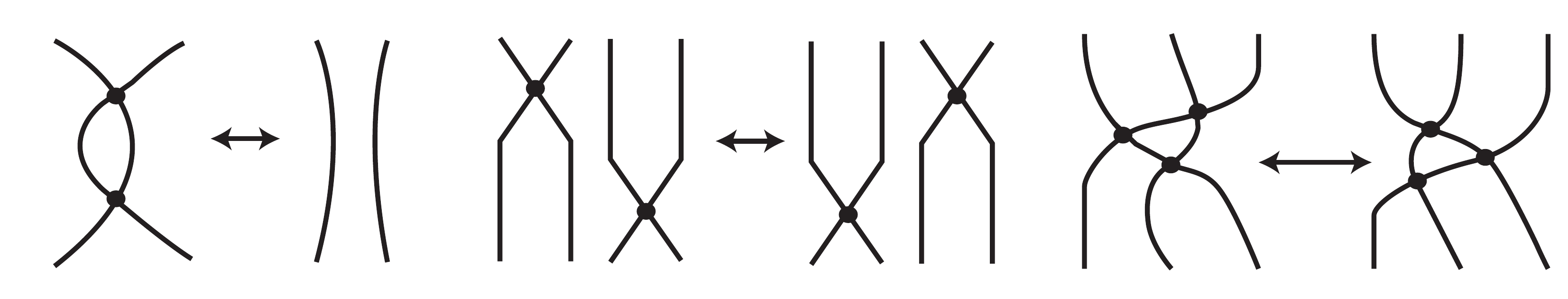}

\end{center}

 \caption{Artin moves for free braids}\label{movesArtin}
\end{figure}
\end{dfn}

\begin{dfn}\cite{Manturov2}
Let $G_{n}^{3}$ be the group given by the presentation generated by $\{ a_{\{ijk\}}~|~ \{i,j,k\} \subset \{1, \dots, n\}, |\{i,j,k\}| = 3\}$ subject to the following relations:

\begin{enumerate}
\item $a_{\{ijk\}}^{2} = 1$ for all $\{i,j,k\} \subset \{1, \cdots,n\}$, 
\item $a_{\{ijk\}}a_{\{stu\}} = a_{\{stu\}}a_{\{ijk\}}$, if $| \{i,j,k\} \cap \{s,t,u\} | < 2$.
\item $a_{\{ijk\}}a_{\{ijl\}}a_{\{ikl\}}a_{\{jkl\}} = a_{\{jkl\}}a_{\{ikl\}}a_{\{ijl\}}a_{\{ijk\}}$ for distinct $i,j,k,l$.
\end{enumerate}
Simply, we denote $a_{\{ijk\}} = a_{ijk}$.
\end{dfn}

There is a homomorphism $ PB_{n} \rightarrow G_{n}^{3}$, which is constructed in~\cite{ManturovNikonov}. In the next section, we shall enhance this homomorphism by adding new structures (parity, dots) to the group $G_{n}^{k}$. Geometrically it originates from homomorphisms $PB_{n} \rightarrow PB_{n+1}$ such that each homomorphism adds a solitary strand.
\begin{dfn}\cite{Kim}
For a positive integer $n$, let us define $G_{n,p}^{2}$ as the group presentation generated by $ \{ a_{ij}^{\epsilon} ~|~ \{i,j\} \subset \{1, \dots, n\}, i < j, ~\epsilon \in \{0,1\}  \}$ subject to the following relations:
\begin{center}\begin{enumerate}
\item $(a_{ij}^{\epsilon})^{2} = 1$, $\epsilon \in \{0,1\}$ and $i,j \in \{1, \cdots, n\}$,
\item $a_{ij}^{\epsilon_{ij}}a_{kl}^{\epsilon_{kl}} = a_{kl}^{\epsilon_{kl}}a_{ij}^{\epsilon_{ij}}$ for $1 \leq i <j<k<l \leq n$,
\item $a_{ij}^{\epsilon_{ij}}a_{ik}^{\epsilon_{ik}}a_{jk}^{\epsilon_{jk}} = a_{jk}^{\epsilon_{jk}}a_{ik}^{\epsilon_{ik}}a_{ij}^{\epsilon_{ij}}$,  for $1 \leq i <j<k \leq n$, where $\epsilon_{ij}+\epsilon_{ik}+\epsilon_{jk} \equiv 0$ mod $2$.
\end{enumerate}
\end{center}
\end{dfn}

\section{Brunnian braids in $PB_{n}$ and in $G_{n}^{3}$}
In classical braid groups, a special role is played by {\em Brunnian braids on $n$ strands}. A Brunnian braid is a braid such that each braid obtained by omitting one strand is trivial for every strand. We can define Brunnian elements in groups $G_{n}^{k}$ and we will see that the image of a Brunnian braid is Brunnian in $G_{n}^{3}$ by a homomorphism from $PB_{n}$ to $G_{n}^{3}$ in \cite{ManturovNikonov}.
Let us define a mapping $p_{m} : PB_{n+1} \rightarrow PB_{n}$ by
\begin{center}
$p_{m}(b_{ij})  = \left\{
\begin{array}{cc} 
    1 & \text{if}~ j= m, \\
        b_{ij} & \text{if}~ i,j<m, \\
       b_{i(j-1)} & \text{if}~ i<m, j>m,\\
       b_{(i-1)(j-1)} & \text{if}~ i, j>m, \\
   \end{array}\right.$
   \end{center}
Roughly speaking, this mapping deletes one strand from braids on $n$ strands.
\begin{figure}[h!]
\begin{center}
 \includegraphics[width =10cm]{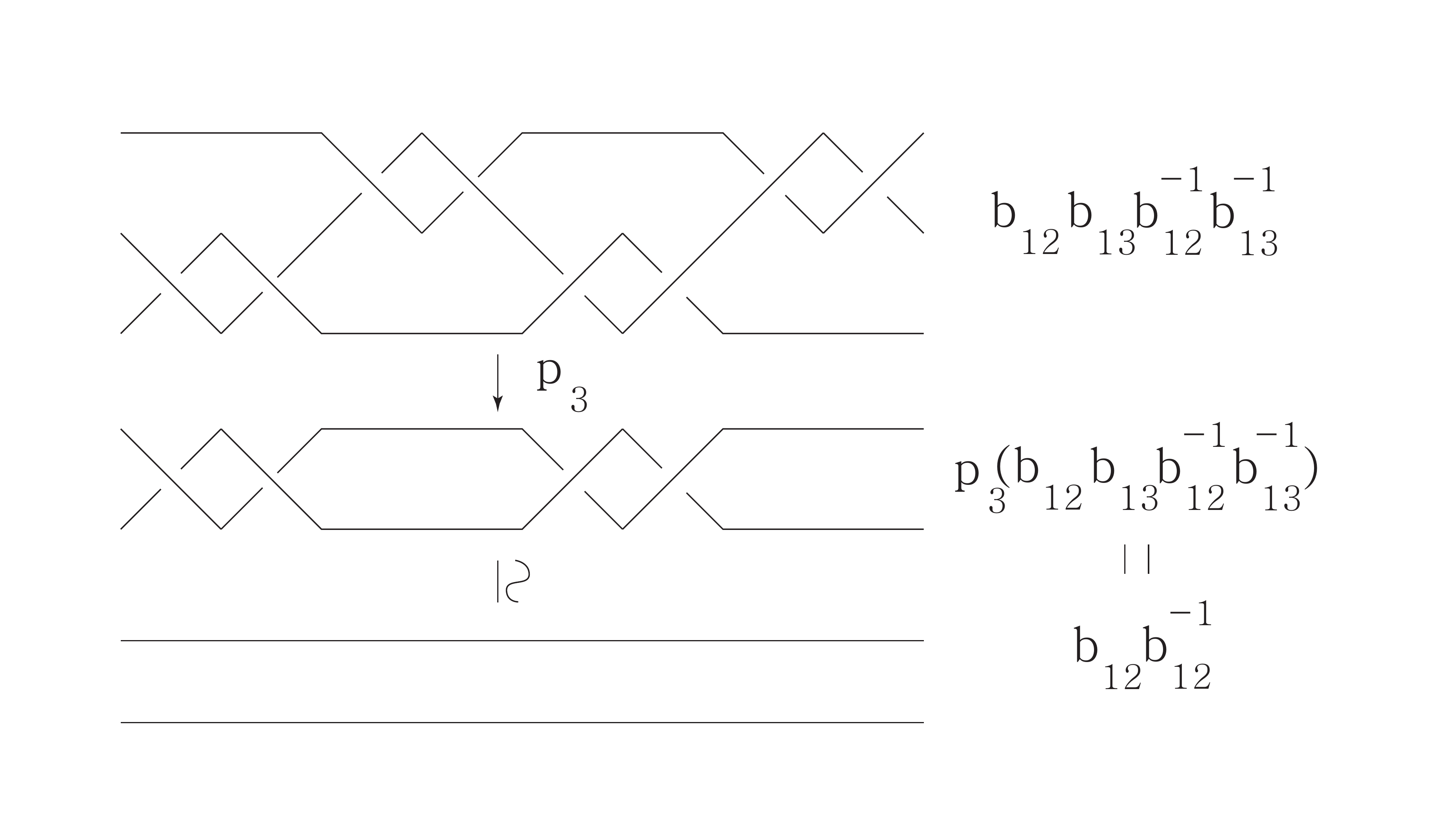}

\end{center}

 \caption{Diagrams of $b_{12}b_{13}b_{12}^{-1}b_{13}^{-1}$ and $p(b_{12}b_{13}b_{12}^{-1}b_{13}^{-1}) = b_{12}b_{12}^{-1}$}\label{exa-projection}
\end{figure}

\begin{dfn}
{\em An $n$-strand Brunnian braid} is a braid on $n$ strands such that $p_{m}(\beta) = 1$ for each index $m$.
\end{dfn}
Note that every element in $PB_{n}$ of the form $[\cdots [[b_{i_{1}j_{1}},b_{i_{2}j_{2}}],b_{i_{3}j_{3}}],\cdots],b_{i_{m}j_{m}}]$ is Brunnian, for example, Fig.~\ref{exa-projection}.

Let us define a mapping $q_{m} : G_{n+1}^{3} \rightarrow G_{n}^{3}$ by

\begin{center}
$q_{m}(a_{ijk})  = \left\{
\begin{array}{cc} 
    1 & \text{if}~ k= m \\
       a_{ijk} & \text{if}~ i,j,k<m, \\
      a_{ij(k-1)} & \text{if}~ i,j<m, k>m, \\
       a_{i(j-1)(k-1)} & \text{if}~ i<m, j, k>m,\\
       a_{(i-1)(j-1)(k-1)} & \text{if}~ i, j, k>m,\\
   \end{array}\right.$
   \end{center}

\begin{dfn}
An element $\beta$ from $G_{n+1}^{3}$ is called {\em Brunnian} if $q_{m}(\beta) = 1$ for each index $m$.
\end{dfn}
In~\cite{ManturovNikonov} the authors constructed homomorphisms from $PB_{n}$ to $G_{n}^{3}$ and to $G_{n}^{4}$. We recall the mapping $\phi_{n}$ from $PB_{n}$ to $G_{n}^3$. Let 
$$c^{n}_{i,j} = \prod_{k=j+1}^{n} a_{ijk} \prod_{k=1}^{j-1} a_{ijk} \in G_{n}^{3}.$$ 
Let us define $\phi_{n} : PB_{n} \rightarrow G_{n}^{3}$ by 
$$\phi_{n}(b_{ij}) = (c^{n}_{i,i+1})^{-1}(c^{n}_{i,i+2})^{-1} \cdots (c^{n}_{i,j-1})^{-1} (c^{n}_{i,j})^{2} c^{n}_{i,j-1} \cdots c^{n}_{i,i+2} c^{n}_{i,i+1},$$ 
for each generator $b_{ij}$.
\begin{prop}\cite{ManturovNikonov}
The mapping $\phi_{n} : PB_{n} \rightarrow G_{n}^{3}$ is well defined.
\end{prop}

\begin{thm}\label{thm_pres_brun}
For a Brunnian braid $\beta \in PB_{n}$, $\phi_{n}(\beta)$ is a Brunnian in $G_{n}^{3}$.
\end{thm}
To prove the above theorem, firstly we prove the following lemma.
\begin{lem}
For $i,j \in  \{1, \cdots ,n \}$ and $i<j$,
\begin{center}
$q_{n}(c^{n}_{i,j})  = \left\{
\begin{array}{cc} 
    1 & \text{if}~ j = n, \\
       c^{n-1}_{i,j} & \text{if}~j \neq n. \\
     \end{array}\right.$
   \end{center}
\end{lem}

\begin{proof}
If $j=n$, then $$q_{n}(c^{n}_{i,n}) = q_{n}(  \prod_{k=1}^{n-1} a_{ink})= \prod_{k=1}^{n-1}p_{n}(a_{ink}) =1.$$
If $i,j \neq n$, then $$q_{n}(c^{n}_{i,j}) = q_{n}(  \prod_{k=j+1}^{n} a_{ijk} \prod_{k=1}^{j-1} a_{ijk})= \prod_{k=j+1}^{n} p_{n}(a_{ijk}) \prod_{k=1}^{j-1}p_{n}(a_{ijk}) = \prod_{k=j+1}^{n-1} a_{ijk} \prod_{k=1}^{j-1} a_{ijk}= c^{n-1}_{i,j}.$$\\
\end{proof}

Analogously we can show that 
\begin{center}
$q_{m}(c_{i,j})  = \left\{
\begin{array}{cc} 
    1 & \text{if}~ j= m \\
        c_{ij} & \text{if}~ i,j<m,\\
      c_{i(j-1)} & \text{if}~ i<m, j>m,\\
       c_{(i-1)(j-1)} & \text{if}~ i, j>m,\\
   \end{array}\right.$
   \end{center}

\begin{proof}[Proof of Theorem~\ref{thm_pres_brun}]
It is sufficient to show that $q_{m} \circ \phi_{n} = \phi_{n-1} \circ p_{m}$, because if $p_{m}(\beta) =1$, then $q_{m} \circ \phi_{n}(\beta) = \phi_{n-1} \circ p_{m}(\beta) = \phi_{n-1}(1)=1$. For $m=n$ and $b_{ij} \in PB_{n}$, if $j=n$, then $\phi_{n-1} \circ p_{n}(b_{in}) = 1$ and

\begin{eqnarray*}
q_{n} \circ \phi_{n} (b_{in}) &=& q_{n}((c^{n}_{i,i+1})^{-1}(c^{n}_{i,i+2})^{-1} \cdots (c^{n}_{i,n-1})^{-1} (c^{n}_{i,n})^{2} c^{n}_{i,n-1} \cdots c^{n}_{i,i+2} c^{n}_{i,i+1}) \\
 &=& (c^{n-1}_{i,i+1})^{-1}(c^{n-1}_{i,i+2})^{-1} \cdots (c^{n-1}_{i,n-1})^{-1} q_{n}((c^{n}_{i,n})^{2}) c^{n-1}_{i,n-1} \cdots c^{n-1}_{i,i+2} c^{n-1}_{i,i+1} \\ 
 &=&  (c^{n-1}_{i,i+1})^{-1}(c^{n-1}_{i,i+2})^{-1} \cdots (c^{n-1}_{i,n-1})^{-1} c^{n-1}_{i,n-1} \cdots c^{n-1}_{i,i+2} c^{n-1}_{i,i+1} =1.
\end{eqnarray*}
 If $i,j \neq n$, then 
 \begin{eqnarray*}
  \phi_{n-1} \circ p_{n}(b_{ij}) &=& \phi_{n-1}(b_{ij}) \\ 
  &=&  (c^{n-1}_{i,i+1})^{-1}(c^{n-1}_{i,i+2})^{-1} \cdots (c^{n-1}_{i,j-1})^{-1} (c^{n-1}_{i,j})^{2} c_{i,j-1} \cdots c^{n-1}_{i,i+2} c^{n-1}_{i,i+1},
 \end{eqnarray*}

 and
 \begin{eqnarray*}
 q_{n} \circ \phi_{n} (b_{ij}) &=& q_{n}((c^{n}_{i,i+1})^{-1}(c^{n}_{i,i+2})^{-1} \cdots (c^{n}_{i,j-1})^{-1} (c^{n}_{i,j})^{2} c^{n}_{i,j-1} \cdots c^{n}_{i,i+2} c^{n}_{i,i+1})\\
 &=& (c^{n-1}_{i,i+1})^{-1}(c^{n-1}_{i,i+2})^{-1} \cdots (c^{n-1}_{i,j-1})^{-1} (c^{n-1}_{i,j})^{2} c^{n-1}_{i,j-1} \cdots c^{n-1}_{i,i+2} c^{n-1}_{i,i+1}.
\end{eqnarray*}

Analgously, it is easy to show that  $q_{m} \circ \phi_{n} = \phi_{n-1} \circ p_{m}$. 
\end{proof}

\section{A new index invariant for $G_{n}^{2}$ and $G_{n}^{3}$}

The group $G_{n}^{3}$ is very powerful. In~\cite{ManturovNikonov},  some index invariants valued in $\mathbb{Z}_{2} * \cdots * \mathbb{Z}_{2}$ were extracted from $G_{n}^{3}$. 
The aim of this section is to show that the MN invariants constructed previously fail to recognize the non-triviality of Brunnian braids(Lemma~\ref{failrecognizebrunnian}). We can enhance these invariants by using the structure of $G_{n}^{k}$; in fact, we shall make only one step allowing us to recognize the commutator, see example \ref{recog_brun}.  But in principle, it is possible to go on enhancing the invariants coming from $G_{n}^{k}$ (even from $G_{n}^{3}$) to get invariants which recognize the non-triviality of commutators of arbitrary lengths: $[[[[[b_{12},b_{13}],b_{14}],b_{15},...]$. More precisely, the group $G_{n}^{3}$ itself recognizes the non-triviality of such braids, and the corresponding invariants can be derived as maps from $G_{n}^{3}$ to free products of $Z_{2}$ (Example~\ref{recog_brun}). It would be interesting to compare it with lower central series\cite{BardakovVershininWu},\cite{BardakovMikhailovVershininWu2},\cite{BardakovMikhailovVershininWu1}.

Firstly, we recall the definition of those homomorphisms(invariants) and construct new ones. 
\begin{dfn}
Let $\beta \in G_{n}^{k}$. If the number of $a_{m}$ of $\beta$ is even for each multiindex $m$, then it is called {\em a word in a good condition}. Analogously if the number of $b_{ij}$ of $\beta \in PB_{n}$ is even for every pair $i,j$, then $\beta$ is called {\em a free braid in a good condition}.
\end{dfn}

\begin{rem}
Let $\beta$ and $\beta'$ in $G_{n}^{k}$ such that $\beta = \beta'$ in $G_{n}^{k}$. Since the relation $a_{m}^{2}$ from $G_{n}^{k}$ changes the number of $a_{m}$ in $\beta$ and other relations do not change the number of $a_{m}$, it is easy to show that if $\beta$ is in good condition, then $\beta'$ is also in good condition. Let $\beta,\beta' \in PB_{n}$ such that $\beta$ and $\beta'$ are equivalent. Analogously, we can show that if $\beta$ is in good condition, then $\beta'$ is also in good condition.
\end{rem}

\begin{rem}
Let $H$ be a subset of all elements in good condition in $PB_{n}$. Then $H$ is a normal subgroup of $PB_{n}$ of finite index. Analogously we can show that a subset of all elements in good condition of $G_{n}^{k}$ is a normal subgroup of $G_{n}^{k}$ of finite index.
\end{rem}
 
We recall the invariant of Manturov and Nikonov, simply MN-invariant, for the case of braids in $G_{n}^{3}$. Let $\beta$ be a free braid on $n$ strands in a good condition. For each $c=a_{ijk}$ of $\beta$ and for $ l \in \{1,2,\cdots, n\} \backslash \{i,j,k\}$, define $i_{c}(l)$ by 
$$i_{c}(l) = (N_{jkl}+N_{ijl}, N_{ikl}+N_{ijl}) \in \mathbb{Z}_{2} \times \mathbb{Z}_{2},$$
where $N_{ikl}$ is the number of $a_{ikl}$ from the start of $\beta$ to the crossing $c$. Note that $i_{c}$ can be considered as a map from $\{1,2,\cdots, n \} \backslash \{i,j,k\} $ to $\mathbb{Z}_{2} \times \mathbb{Z}_{2}$. Fix $i,j,k\in \{1,\dots, n\}$. Let $\{c_{1}, \cdots, c_{m}\}$ be the set of $a_{ijk}$ such that for each $s , t \in \{1, 2, \cdots, m\}$, $ s< t$ if and only if we meet $c_{s}$ earlier than $c_{t}$ in $\beta$. Define a group presentation $F_{n}^{3}$ generated by $\{ \sigma ~|~ \sigma : \{1,2,\cdots n\} \backslash \{i,j,k\} \rightarrow \mathbb{Z}_{2} \times \mathbb{Z}_{2} \}$ with relations $\{ \sigma^{2} = 1\}$. Note that $i_{c}$ is a mapping from $\{1,2,\cdots n\} \backslash \{i,j,k\} \rightarrow \mathbb{Z}_{2} \times \mathbb{Z}_{2}$ and $i_{c}$ is in  $F_{n}^{3}$. In other words, we deal with a free product of $2^{2(n-3)}$ copies of $\mathbb{Z}_{2}$. Define a word $w_{(i,j,k)}(\beta)$ in $F_{n}^{3}$ for $\beta$ by $w_{(i,j,k)}(\beta) = i_{c_{1}}i_{c_{2}} \cdots i_{c_{m}}$.
 
\begin{prop}\cite{ManturovNikonov}
For a positive integer $n$ and for $i, j,k \in \{1, \cdots ,n\}$ such that $|\{i,j,k \} | =3$, $w_{(i,j,k)}$ is an invariant for $H$ of $G_{n}^{3}$ in a good condition.

\end{prop}

\begin{lem}\label{failrecognizebrunnian}
For a Brunnian $\beta \in PB_{n}$, $w_{(i,j,k)}(\phi_{n}(\beta)) =1$.
\end{lem}

\begin{proof}
It is sufficient to show that $w_{(ijk)}(\phi_{n}(b_{lm})) =1$ for $|\{l,m \} \cap \{i,j,k\}|<2$.
For a Brunnian braid $\beta \in PB_{n}$ and for $l \not\in \{i,j,k\}$, let $\beta_{l}$ be a braid obtained by omitting $b_{lm}$ from $\beta$. Since $w_{(ijk)}(\phi_{n}(b_{lm})) =1$ for $|\{l,m \} \cap \{i,j,k\}|<2$,
$$w_{(ijk)}(\phi_{n}(\beta)) = w_{(ijk)}(\phi_{n}(\beta_{l})).$$
Since $\beta$ is Brunnian, $\beta_{l}$ is trivial and $w_{(ijk)}(\phi_{n}(\beta)) = w_{(ijk)}(\phi_{n}(\beta_{l})) =1.$ Now we will show the statement is true. By the definition of $\phi_{n}$,
$$\phi_{n}(b_{ij}) = (c^{n}_{i,i+1})^{-1}(c^{n}_{i,i+2})^{-1} \cdots (c^{n}_{i,j-1})^{-1} (c^{n}_{i,j})^{2} c^{n}_{i,j-1} \cdots c^{n}_{i,i+2} c^{n}_{i,i+1}.$$  
Notice that for $|\{l,m\} \cap \{i,j,k\}| \neq 2$, $c^{n}_{l,m}$ contains no $a_{ijk}$. There are 9 subcases:
\begin{enumerate}
\item $\{ l, m \} \cap \{i,j,k\} = \emptyset$,
\item $l=i, 0<m<i$, 
\item  $l=i, i<m<j$,
\item $l=i, j<m<k$,
\item $l=i, k<m<n$,
\item $l=j, j<m<k$,
\item  $l=j, k<m<n$,
\item   $l=k, k<l<n$. 
\end{enumerate}
If $\{ l, m \} \cap \{i,j,k\} = \emptyset$, then $\phi_{n}(b_{ij})$ has no $c^{n}_{i,j}$, $c^{n}_{i,k}$ and $c^{n}_{j,k}$. Hence $w_{(ijk)}(\phi_{n}(b_{lm})) =1$. Analogously  $w_{(ijk)}(\phi_{n}(b_{lm})) =1$ in the case of (2),(3),(6) and (8). \\
If $l=i, j<m<k$, then
$$\phi_{n}(b_{im}) = (c^{n}_{i,i+1})^{-1} \cdots (c^{n}_{i,j})^{-1} \cdots (c^{n}_{i,j-1})^{-1} (c^{n}_{i,m})^{2} c^{n}_{i,j-1} \cdots c^{n}_{i,j} \cdots c^{n}_{i,i+1},$$ 
has just two $a_{ijk}$, say $c_{1}=a_{ijk}$ in $(c^{n}_{i,j})^{-1}$ and  $c_{2}=a_{ijk}$ in $c^{n}_{i,j}$, respectively. Since the number of each $a_{stu}$ between $c_{1}$ and $c_{2}$ is even for every $s,t,u \in \{1, \cdots n\}$, $i_{c_{1}} = i_{c_{2}}$. Therefore $w_{(ijk)}(\phi_{n}(b_{lm})) =1$. Analogously $w_{(ijk)}(\phi_{n}(b_{lm})) =1$ in the cases of (4) and (7). \\
If $l=i, k<m<n$, 
$$\phi_{n}(b_{im}) = (c^{n}_{i,i+1})^{-1} \cdots (c^{n}_{i,j})^{-1} \cdots (c^{n}_{i,k})^{-1} \cdots (c^{n}_{i,m})^{2} \cdots c^{n}_{i,k} \cdots c^{n}_{i,j} \cdots c^{n}_{i,i+1},$$ 
and it has four $a_{ijk}$, say $c_{1} = a_{ijk}$ in $(c^{n}_{i,j})^{-1}$, $c_{2} = a_{ijk}$ in $(c^{n}_{i,k})^{-1}$, $c_{3} = a_{ijk}$ in $c^{n}_{i,k}$ and $c_{4} = a_{ijk}$ in $c^{n}_{i,j}$, respectively. Since the number of each $a_{stu}$ between $c_{2}$ and $c_{3}$ is even for every $s,t,u \in \{1, \cdots n\}$, $i_{c_{2}} = i_{c_{3}}$. Similarly, $i_{c_{1}} = i_{c_{4}}$. Therefore $w_{ijk}(\phi_{n}(b_{im})) = i_{c_{1}} i_{c_{2}} i_{c_{3}} i_{c_{4}}=1$.
\end{proof}

By the above Lemma the MN-invariant for $G_{n}^{3}$ does not recognize the non-triviality of Brunnian braids in $PB_{n}$. Now we have parity for $G_{n}^{2}$ and we can extend MN-invariant for $G_{n}^{2}$ by using parity. To make things clearer, let us first consider the case of $G_{n}^{2}$. Let $\beta$ be a free braid on $n$ strands in a good condition. For each classical crossing $c$ of $\beta$ of type $(i,j)$ and for $ k \in \{1,2,\cdots, n\} \backslash \{i,j\}$, define $i^{i}_{c}(k)$ by the sum of number of all crossings of type $(i,k)$ from the start of $i$-th strand to the crossing $c$. Set $i_{c}(k) = i^{i}_{c}(k) +i^{j}_{c}(k)$ modulo $\mathbb{Z}_{2}$, for example, see Fig.~\ref{exa_ic}.
\begin{figure}[h!]
\begin{center}
 \includegraphics[width =10cm]{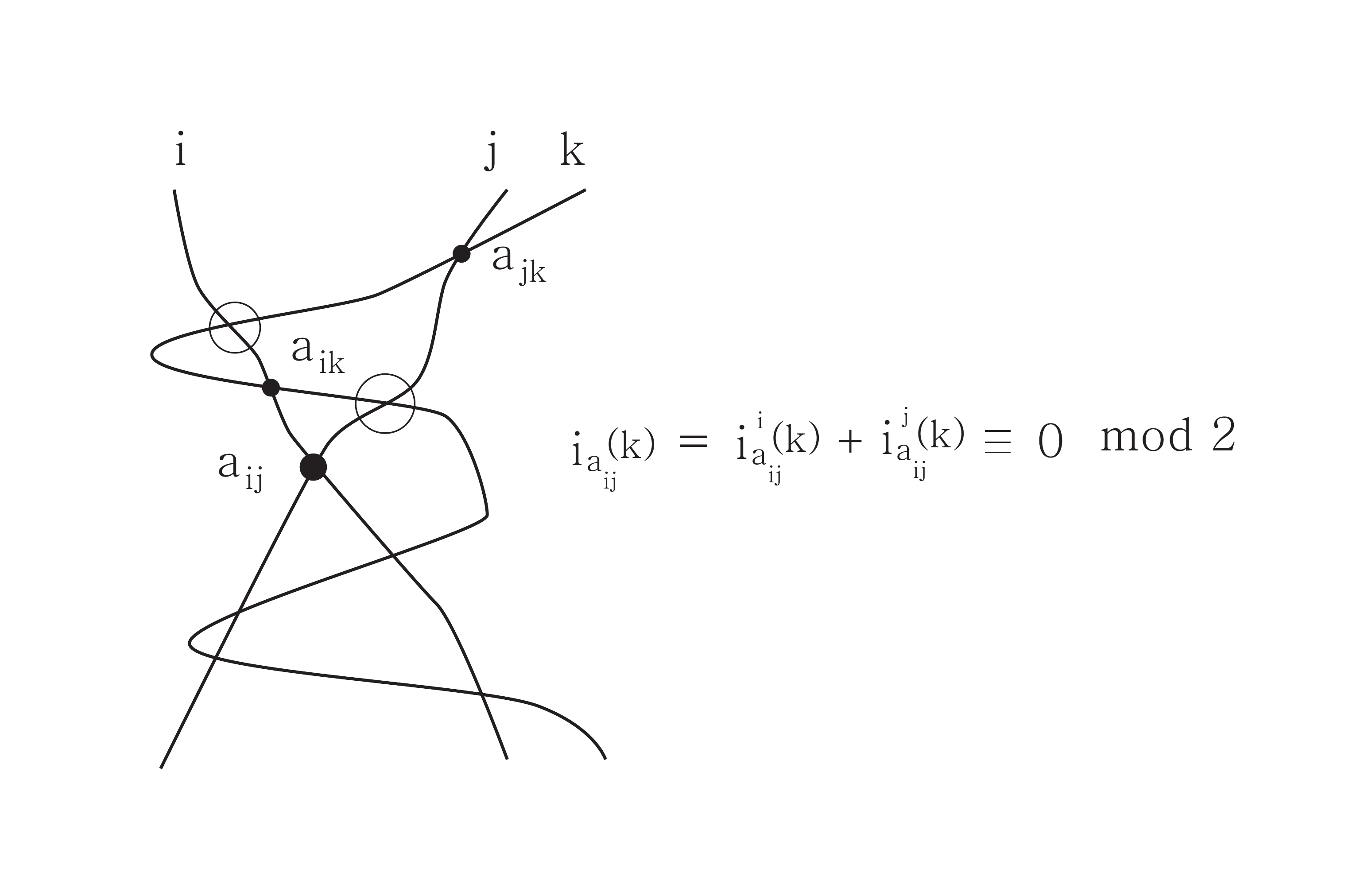}

\end{center}

 \caption{The value of $i_{a_{ij}}(k)$}\label{exa_ic}
\end{figure}
Note that $i_{c}$ can be considered as a map from $\{1,2,\cdots, n \} \backslash \{i,j\} $ to $\mathbb{Z}_{2}$.
 Fix $i,j\in \{1,\dots, n\}$ such that $i \neq j$. Let $\{c_{1}, \cdots, c_{m}\}$ be the set of classical crossings of type $(i,j)$ such that for each $k , l \in \{1, 2, \cdots, m\}$, $k < l$ if and only if we meet $c_{k}$ earlier than $c_{l}$ in $\beta$. 
  Let us consider $F_{n}^{2}$ the free product of groups $\mathbb{Z}_{2}$ generated by $\{ \sigma ~|~ \sigma : \{1,2,\cdots n\} \backslash \{i,j\} \rightarrow \mathbb{Z}_{2} \}$ with relations $\{ \sigma^{2} = 1\}$. Note that $i_{c}$ is a mapping from $\{1,2,\cdots n\} \backslash \{i,j\} \rightarrow \mathbb{Z}_{2}$ and $i_{c}$ is in  $F_{n}^{2}$. In other words, we deal with a free product of $2^{(n-2)}$ copies of $\mathbb{Z}_{2}$. Define a word $w_{(i,j)}(\beta)$ in $F_{n}^{2}$ for $\beta$ by $w_{(i,j)}(\beta) = i_{c_{1}}i_{c_{2}} \cdots i_{c_{m}}$.
 
\begin{prop}\cite{ManturovNikonov}
For a positive integer $n$ and for $i, j \in \{1, \cdots ,n\}$ such that $ i \neq j$, $w_{(i,j)}$ is an invariant for oriented enumerated free braids in a good condition.

\end{prop}

Let $\beta \in G_{n,p}^{2}$. For fixed pair $i,j \in \{1, \cdots, n\}$ and for $k \in \{1, \cdots, n\} \backslash \{i,j\}$, define $i_{a_{ij}^{\epsilon}}^{p}(k)$ for each $a_{ij}^{\epsilon}$ in $\beta$ by 

\begin{center}
$ i_{a_{ij}^{\epsilon}}^{p}(k) = \left \{
\begin{array}{cc} 
    N_{ik}^{0}+N_{jk}^{0} ~ mod ~2 & \text{if}~\epsilon =0, \\
    N_{ik}^{0}+N_{jk}^{1} ~ mod~ 2 &  \text{if}~\epsilon =1. \\
   \end{array}\right. $
\end{center}
where $N_{ik}^{\epsilon}$ is the number of $a_{ik}^{\epsilon}$, which appears before $a_{ij}^{\epsilon}$, for example, see Fig.~\ref{exa_ic_parity}.
\begin{figure}[h!]
\begin{center}
 \includegraphics[width =10cm]{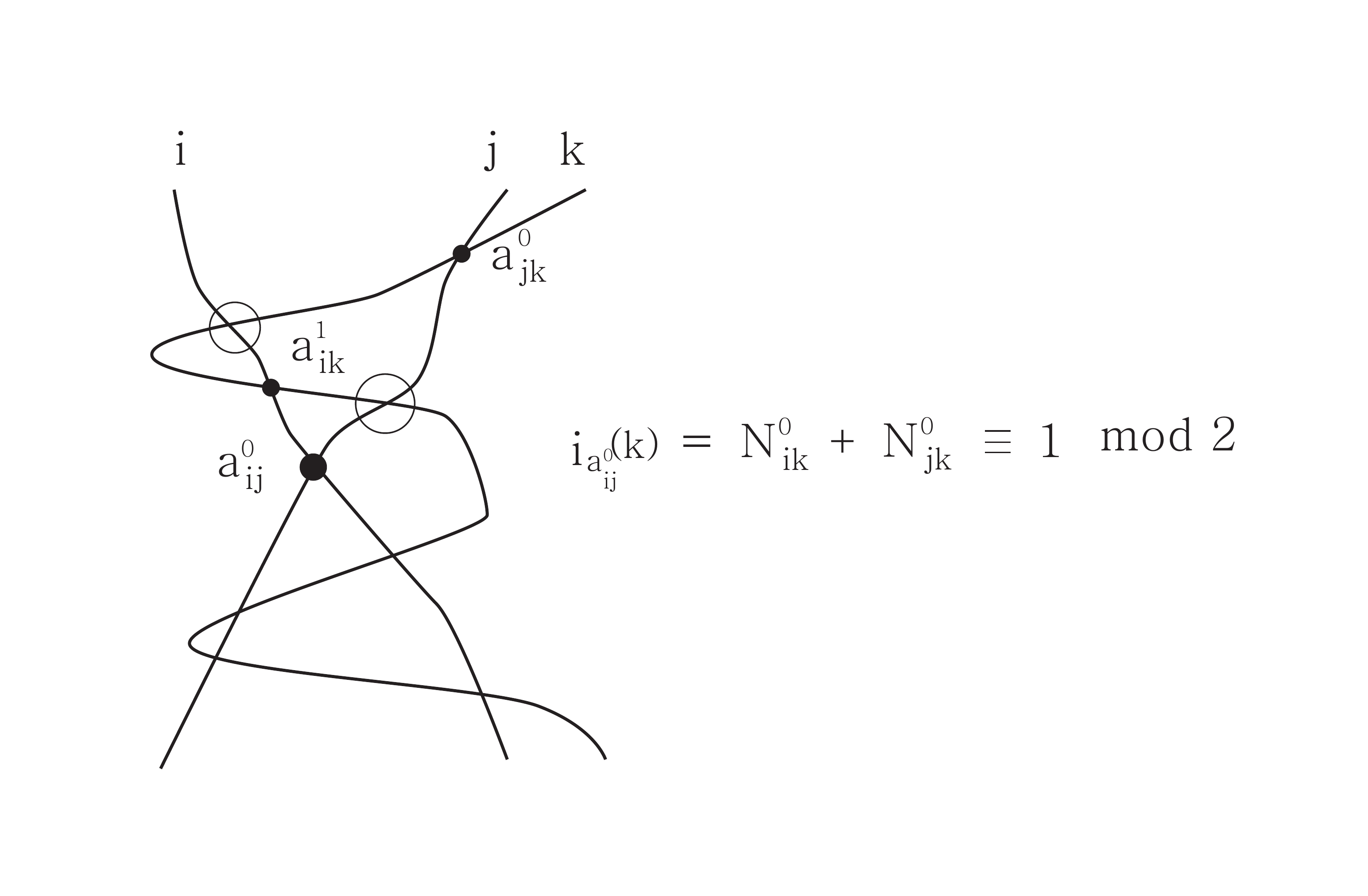}

\end{center}

 \caption{The value of $i^{p}_{a_{ij}}(k)$}\label{exa_ic_parity}
\end{figure}

Let $\{c_{1}, \cdots c_{m}\}$ be the ordered set of $a_{ij}^{\epsilon_{m}}$'s such that the order agrees with the order of position of $a_{ij}^{\epsilon}$'s. Define $w^{p}_{ij} : G_{n,p}^{2} \rightarrow F_{n}^{2}$ by $w^{p}_{ij}(\beta) =\prod_{s=1}^{m} i_{c_{s}}^{p} $.(The superscript $p$ means `parity'.)
\begin{lem}\label{invGnp2}
$w^{p}_{ij}$ is well defined.
\end{lem}

\begin{proof}
It suffices to show that the image of $w^{p}_{ij}$ does not change when relations of $G_{n,p}^{2}$ are applied to $\beta$. In the cases of $(a_{ij}^{\epsilon})^{2}=1$ and $a_{ij}^{\epsilon_{1}}a_{kl}^{\epsilon_{2}}= a_{kl}^{\epsilon_{2}}a_{ij}^{\epsilon_{1}}$, it is easy. For relations $a_{ij}^{\epsilon_{ij}}a_{ik}^{\epsilon_{ik}}a_{jk}^{\epsilon_{jk}} = a_{jk}^{\epsilon_{jk}}a_{ik}^{\epsilon_{ik}}a_{ij}^{\epsilon_{ij}}$, where $\epsilon_{ij}+\epsilon_{ik}+\epsilon_{jk}=0$ mod $2$, suppose that the relation is not applied on $c_{s}$. Then the number of $a_{ik}^{\epsilon}$ and $a_{jk}^{\epsilon}$ before $c_{s}$ remains and then $w^{p}_{ij}(\beta)$ does not change. Suppose that $c_{s}$ is in the applied relation, say $c_{s}^{\epsilon_{1}}a_{il}^{\epsilon_{2}}a_{jl}^{\epsilon_{3}} = a_{jl}^{\epsilon_{3}}a_{il}^{\epsilon_{2}}c_{s}^{\epsilon_{1}}$, where ${\epsilon_{1}}+{\epsilon_{2}}+{\epsilon_{3}}=0$ mod $2$. If $l \neq k$, then  the number of $a_{ik}^{\epsilon}$ and $a_{jk}^{\epsilon}$ before $c_{s}$ is not changed. Suppose that $l=k$. If $\epsilon_{1}=0$, then $\epsilon_{2} = \epsilon_{3} =0 $ or  $\epsilon_{2} = \epsilon_{3} =1$. Then the sum of the number of $a_{ik}^{0}$ and $a_{jk}^{0}$ remains modulo $2$. If $\epsilon_{1}=1$, then  $\epsilon_{2} = 1,\epsilon_{3} =0 $ or  $\epsilon_{2} = 0$,$\epsilon_{3} =1$. Then $i_{k}^{p}(c_{s}) = N_{ik}^{0}+N_{jk}^{1} $ is not changed modulo $2$ and $w^{p}_{ij}(\beta)$ is not changed. Therefore $w^{p}_{ij}$ is an invariant.
\end{proof}
Let $H_{n+1}^{2}$ be a subgroup of all elements in good condition in $G_{n+1}^{2}$. Now we define a mapping $\phi_{k}$ from $H_{n+1}^{2}$ to $G_{n,p}^{2}$ for a fixed $k \in \{1, \cdots n+1 \}$. Roughly speaking, $\phi_{k}$ deletes $k$-th strands from $\beta$, counts how many times $k$-th strand is linked with $i_{l}$-the and $j_{l}$-th strands before $a_{i_{l}j_{l}}$ in $\beta$ and takes the sum of them. In detail, $\phi_{k}$ is defined as followings. Let $\beta = \prod_{l=1}^{m} a_{i_{l}j_{l}}$. Denote $b_{i_{l}j_{l}} $ by 
\begin{center}
$b_{i_{l}j_{l}} = \left \{
\begin{array}{cc} 
    1 & \text{if}~ k \in \{i_{l},j_{l}\} , \\
     a_{i_{l}j_{l}}^{\epsilon_{l}} &  \text{if}~k \notin \{i_{l},j_{l}\}. \\
   \end{array}\right. $
\end{center}
where $\epsilon_{l}$ is the number of $a_{j_{l}k}$ and $a_{i_{l} k }$ in $\beta$ before $ a_{i_{l}j_{l}}$ modulo $2$.
Now, define $\psi_{k}(\beta) = \prod_{l=1}^{m} b_{i_{l}j_{l}}$.

\begin{prop}\cite{Kim}
$\psi_{k}$ is well defined.
\end{prop}

\begin{dfn}
Let $\beta \in G_{n+1}^{2}$ be in good condition. Then $w_{ij}^{l} : G_{n+1}^{2} \rightarrow F_{n}^{2}$ is defined by $w_{ij}^{l} = w_{ij}^{p} \circ \psi_{l}$.
\end{dfn}
For example, for $\beta = a_{12}a_{34}a_{13}a_{34}a_{13}a_{12}$, we obtain $\psi_{4}(\beta) = a_{12}^{0}a_{13}^{1}a_{13}^{0}a_{12}^{0}$. Then $w_{12}^{4}(\beta) = w_{12}^{p} \circ \psi_{4}(\beta) = 01 \in F_{3}^{2}$, see Fig.~\ref{exa_psi_parity_inv}.
\begin{figure}[h!]
\begin{center}
 \includegraphics[width =10cm]{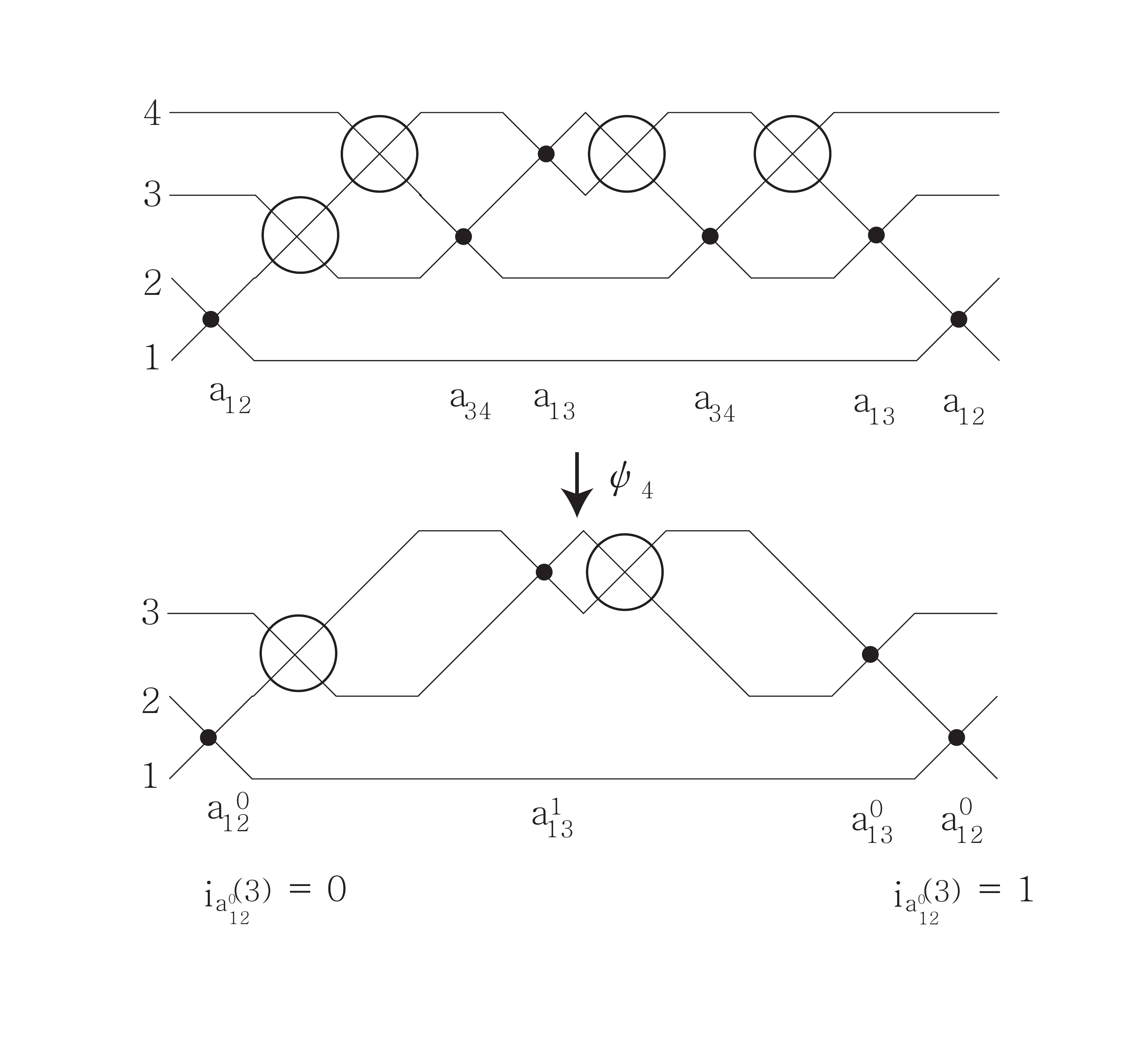}

\end{center}

 \caption{$\beta = a_{12}a_{34}a_{13}a_{34}a_{13}a_{12}$ and  $w_{12}^{4} = 01 \in F_{3}^{2}$ }\label{exa_psi_parity_inv}
\end{figure}

\begin{cor}
$w_{ij}^{l}$ is an invariant for $\beta \in G_{n+1}^{2}$. 
\end{cor}

\begin{proof}
Since $\psi_{l}$ are homomorphism, by Lemma~\ref{invGnp2}, $w_{ij}^{l}$ is an invariant for $\beta \in G_{n+1}^{2}$.
\end{proof}

\begin{exa}
Let $X =  a_{12}a_{13}a_{12}a_{13}$ and $Y = a_{23}a_{35}a_{23}a_{35}$ and 
$\beta = [X,Y]$ in $G_{5}^{2}$. Let us show now that $\beta$ is not trivial. To this end, we consider the element $\omega_{12}^{5}$ where the parity is obtained from $5$-th strand. For a pair $(1,2)$, the value of MN-invariant for $G_{n}^{2}$ of $\beta$ is trivial, because $Y$ is in a good condition and $Y$ contains no $a_{12}$. But $w_{12}^{5}(\beta)$ is not trivial. Now we calculate it. Firstly, $$\psi_{5}(\beta) = a_{12}^{0}a_{13}^{0}a_{12}^{0}a_{13}^{0}a_{23}^{0}a_{23}^{1}a_{13}^{0}a_{12}^{0}a_{13}^{0}a_{12}^{0}a_{23}^{1}a_{23}^{0}.$$ Let $c_{1} = a_{12}^{0}$, $c_{2} = a_{12}^{0}$,  $c_{3} = a_{12}^{0}$,  $c_{4} = a_{12}^{0}$ such that the order of $c_{i}$ agrees with the order of $a_{12}^{\epsilon}$ in $\beta$. Then $i_{c_{1}}^{p}(3) = 0$, $i_{c_{2}}^{p}(3) = 1$, $i_{c_{3}}^{p}(3) = 0$, $i_{c_{4}}^{p}(3) = 1$ and $i_{c_{s}}(4)= 0$. Then $w_{12}^{5}(\beta) = \zeta_{(0,0)}\zeta_{(1,0)}\zeta_{(0,0)}\zeta_{(1,0)}$ and it cannot be canceled in $F_{4}^{2}$, where $\zeta_{a,b}$ is defined by $\zeta_{a,b}(3)=a$ and $\zeta_{a,b}(4)=b$. 
\end{exa}

This invariant can be used for $\beta \in G_{n+1}^{3}$ by homomorphism $r_{m} : G_{n+1}^{3} \rightarrow G_{n}^{2}$ defined by
\begin{center}
$ r_{m}(a_{ijk}) = \left \{
\begin{array}{cc} 
   a_{ij} & \text{if}~k=m,i,j<m, \\
   a_{i(j-1)}  & \text{if}~k=m,i<m,j>m, \\
   a_{(i-1)(j-1)}  & \text{if}~k=m,i>m,j>m, \\
    1 &  \text{if}~i,j,k \neq m. \\
   \end{array} \right. $
\end{center}

\begin{exa}
Let 
\begin{eqnarray*}
\beta  &=&a_{124}a_{123}a_{135}a_{134}a_{124}a_{134}a_{135}a_{123}a_{134}a_{135}a_{134}a_{123}a_{135}a_{134}a_{124}a_{134}a_{135}a_{123}\\
&&a_{124}a_{134}a_{135}a_{134},
\end{eqnarray*}
in $G_{5}^{3}$. Then
$$\beta_{1} = r_{1}(\beta) =a_{24}a_{23}a_{35}a_{34}a_{24}a_{34}a_{35}a_{23}a_{34}a_{35}a_{34}a_{23}a_{35}a_{34}a_{24}a_{34}a_{35}a_{23}a_{24}a_{34}a_{35}a_{34}.$$
For index $5$, \\
$$\psi_{5}(\beta_{1}) 
 = a_{24}^{0}a_{23}^{0}a_{34}^{1}a_{24}^{0}a_{34}^{1}a_{23}^{0}a_{34}^{0}a_{34}^{1}a_{23}^{1}a_{34}^{0}a_{24}^{0}a_{34}^{0}a_{23}^{1}a_{24}^{0}a_{34}^{1}a_{34}^{0}.$$
Let $c_{1} = a_{24}^{0}$, $c_{2} = a_{24}^{0}$, $c_{3} = a_{24}^{0}$ and $c_{4} = a_{24}^{0}$ such that the order of $c_{i}$ agrees with the order of $a_{24}^{\epsilon}$ in $\psi_{5}(\beta_{1}) $.
Then \begin{itemize}
\item $i_{c_{1}}(3)= N_{23}^{0}+N_{23}^{0} = 0+0 =0, mod~2$
\item $i_{c_{2}}(3)= N_{23}^{0}+N_{34}^{0}=1+0 =1,mod~2$
\item $i_{c_{3}}(3)= N_{23}^{0}+N_{34}^{0}=2+2 =0,mod~2$
\item $i_{c_{4}}(3)= N_{23}^{0}+N_{34}^{0} = 2+3 =1.mod~2$
\end{itemize}
Therefore $w_{24}^{5}(\beta_{1}) = 0101 \neq 1$ and hence $\beta$ is not trivial in $G_{5}^{3}$. 
\end{exa}

\begin{exa}\label{recog_brun} Let
\begin{eqnarray*}
\beta  &=& [[[b_{12},b_{14}],b_{16}] ,[b_{13},b_{15}]] \\ 
&=& b_{12} b_{14} b_{12}^{-1} b_{14}^{-1} b_{16} b_{14} b_{12} b_{14}^{-1} b_{12}^{-1} b_{16}^{-1}
 b_{13} b_{15} b_{13}^{-1} b_{15}^{-1}b_{16} b_{12} b_{14} b_{12}^{-1} b_{14}^{-1} b_{16}^{-1} b_{14} b_{12} b_{14}^{-1} \\
 &&b_{12}^{-1}b_{15} b_{13} b_{15}^{-1} b_{13}^{-1} \in G_{6}^{3}.
\end{eqnarray*}
  Note that for each $k \in \{ 1,2,3,4,5,6\}$, $p_{k}(\beta) =1$, that is, $\beta$ is Brunnian in $PB_{n}$. Then\\
 $\psi_{6}(r_{1}(\phi_{6}(\beta))) = a_{23}^{0}a_{24}^{0}a_{25}^{0}a_{23}^{1}a_{35}^{1}a_{34}^{1} a_{45}^{0}a_{24}^{0}a_{34}^{0}a_{45}^{1}a_{24}^{1}a_{35}^{1}a_{23}^{1}a_{25}^{0}a_{24}^{0}a_{23}^{0}\\
a_{25}^{1}a_{24}^{1}a_{24}^{0}a_{25}^{1}a_{23}^{0}a_{24}^{0}a_{25}^{0}a_{23}^{1}a_{35}^{1}a_{24}^{1}a_{45}^{1}a_{34}^{0}a_{24}^{0}a_{45}^{0}a_{34}^{1}a_{35}^{1}a_{23}^{1}a_{25}^{0}a_{24}^{0}a_{23}^{0}a_{25}^{1}\\
a_{24}^{1}a_{35}^{1}a_{24}^{1}a_{35}^{1}a_{25}^{0}a_{25}^{1}a_{35}^{1}a_{24}^{1}a_{35}^{1}a_{23}^{1}a_{34}^{0}a_{35}^{0}a_{23}^{0}a_{34}^{1}a_{24}^{1}a_{45}^{1}a_{25}^{0}a_{35}^{0}a_{45}^{0}a_{25}^{1}a_{35}^{1}\\
a_{24}^{1}a_{34}^{1} a_{23}^{0}a_{35}^{0}a_{24}^{0}a_{35}^{0}a_{25}^{1}a_{45}^{0}a_{35}^{1}a_{25}^{0}a_{45}^{1}a_{24}^{1}a_{35}^{0}a_{24}^{1}a_{25}^{1}a_{23}^{1}a_{24}^{0}a_{25}^{0}a_{23}^{0}a_{35}^{0}a_{34}^{0}\\
a_{45}^{0}a_{24}^{0}a_{34}^{1}a_{45}^{1}a_{24}^{1}a_{35}^{0}a_{23}^{0}a_{25}^{0}a_{24}^{0}a_{23}^{1}a_{25}^{1}a_{24}^{1}a_{35}^{0}a_{24}^{1}a_{45}^{1}a_{34}^{1}a_{24}^{0}a_{45}^{0}a_{34}^{0}a_{24}^{1}a_{35}^{0}\\
a_{25}^{1}a_{25}^{1}a_{35}^{0}a_{24}^{1}a_{34}^{0}a_{45}^{0}a_{24}^{0}a_{34}^{1}a_{45}^{1}a_{24}^{1}a_{35}^{0}a_{24}^{1}a_{25}^{1}a_{23}^{1}a_{24}^{0}a_{25}^{0}a_{23}^{0}a_{35}^{0}a_{24}^{1}a_{45}^{1}a_{34}^{1}\\
a_{24}^{0}a_{45}^{0}a_{34}^{0}a_{35}^{0}a_{23}^{0}a_{25}^{0}a_{24}^{0}a_{23}^{1}a_{25}^{1}a_{24}^{1}a_{35}^{0}a_{24}^{1}a_{45}^{1}a_{25}^{0}a_{35}^{1}a_{45}^{0}a_{25}^{1}a_{35}^{0}a_{24}^{1}a_{35}^{0}a_{23}^{0}\\
a_{34}^{1}a_{35}^{1}a_{23}^{1}a_{34}^{0}a_{24}^{1}a_{35}^{0}a_{25}^{1}a_{45}^{0}a_{35}^{1}a_{25}^{0}a_{45}^{1}a_{24}^{1}a_{34}^{0}a_{23}^{1}a_{24}^{1}a_{25}^{1}$\\
and there are 40 $a_{24}^{\epsilon}$. We obtain that 
$$w_{24}^{6} = \zeta_{(0,0)}\zeta_{(0,1)}\zeta_{(1,1)}\zeta_{(0,0)}\zeta_{(0,1)}\zeta_{(1,1)}\zeta_{(0,0)}\zeta_{(0,1)}\zeta_{(0,0)}\zeta_{(0,1)}\zeta_{(1,1)}\zeta_{(0,1)},$$
where $\zeta_{a,b}$ is defined by $\zeta_{a,b}(3)=a$ and $\zeta_{a,b}(5)=b$ and $w_{24}^{p}$ is not trivial in $F_{5}^{2}$. 

\end{exa}


\section{A map from $G_{n+1}^{3}$ to $G_{n,p}^{3}$}
The groups $G_{n}^{k}$, for $k \geq 3,$ must have plentiful information, which can be used for braids. Moreover, it might be possible to define invariants for $G_{n}^{k}$ valued in the free product of $\mathbb{Z}_{2}$ with respect to a `fixed' index. In this section, we define $G_{n,p}^{3}$, which is called {\em $G_{n}^{3}$ with parity}. We construct homomorphism from $G_{n,p}^{3}$ to the free product of $\mathbb{Z}_{2}$, which generates the invariant.

\begin{dfn}
Let $G_{n,p}^{3}$ be a group generated by $\{ a_{ijk}^{\epsilon} | \{i,j,k\} \subset \{1, \cdots, n\}, |\{i,j,k\}|=3, \epsilon \in \{0,1 \} \}$ subject to the following relations:
\begin{enumerate}
\item $(a_{ijk}^{\epsilon})^{2} =1$,
\item $(a_{ijk}^{\epsilon_{1}}a_{lmn}^{\epsilon_{2}})^{2}=1$, if $| \{i,j,k\} \cap \{l,m,n\}  |<2$ for arbitrarily chosen epsilons,
\item $(a_{ijk}^{\epsilon_{l}}a_{ijl}^{\epsilon_{k}}a_{ikl}^{\epsilon_{j}}a_{jkl}^{\epsilon_{i}})^{2} =1$, where $t \geq i,j,k,l$ and $\sum_{s \in \{i,j,k,l \}  \backslash \{t\}}\epsilon_{s} = 0$ mod $2$.
\end{enumerate}

\end{dfn}

Now we define a mapping $f$ from $G_{n+1}^{3}$ to $G_{n,p}^{3}$ as follows : \\
Let $\beta = \prod_{l=1}^{m} a_{i_{l}j_{l}k_{l}}$. For each $a_{i_{l}j_{l}k_{l}}$, denote $b_{i_{l}j_{l}k_{l}} $as follows: 
\begin{center}
$b_{i_{l}j_{l}k_{l}} = \left \{
\begin{array}{cc} 
    1 & \text{if}~n+1 \in \{i_{l},j_{l},k_{l}\} , \\
     a_{i_{l}j_{l}k_{l}}^{\epsilon_{l}} &  \text{if}~n+1 \notin \{i_{l},j_{l},k_{l}\}, \\
   \end{array}\right. $
\end{center}
where $\epsilon_{l}$ is the sum of numbers of $a_{j_{l}k_{l}(n+1)}$ and $a_{i_{l}k_{l}(n+1)}$ in $\beta$ before $ a_{i_{l}j_{l}k_{l}}$ modulo $2$. This is derived from MN-invariant for $G_{n}^{3}$ : 
$$N_{jkl} + N_{ijl} + N_{ikl}+N_{ijl} = N_{jkl} + N_{ikl}~ \text{mod}~2.$$
Now, define $f(\beta) = \prod_{l=1}^{m} b_{i_{l}j_{l}k_{l}}$ .
\begin{lem}
The mapping $f : G_{n+1}^{3} \rightarrow G_{n,p}^{3}$ is well defined.
\end{lem}

\begin{proof}
We will show that if $\beta'$ is obtained from $\beta$ by applying one of the relations of $G_{n+1}^{3}$, then $f(\beta)$ and $f(\beta ')$ are equivalent in $G_{n,p}^{3}$. In the case of relations $a_{ijk}^{2}= 1$ and $(a_{ijk}a_{lmn})^{2}=1$, it is easy. Now we will show that if $\beta '$ is obtained by applying the relation $a_{ijk}a_{ijl}a_{ikl}a_{jkl} = a_{jkl}a_{ikl}a_{ijl}a_{ijk}$, then $f(\beta)$ and $f(\beta ')$ are equivalent in $G_{n,p}^{3}$, say $\beta = Fa_{ijk}a_{ijl}a_{ikl}a_{jkl}B$ and $\beta'= Fa_{jkl}a_{ikl}a_{ijl}a_{ijk}B$. Suppose that $l =n+1$. Then 
$f(\beta) = F' a_{ijk}^{\epsilon} B'$ and $f(\beta') = F' a_{ijk}^{\epsilon'} B'$. By definition of $f$, $\epsilon = \epsilon' + \alpha$, where $\alpha$ is the number of $a_{ikl}$ and $a_{jkl}$ in $a_{jkl}a_{ikl}a_{ijl}$ and $\alpha \equiv 0$ mod $2$. Suppose that $l \neq n+1$. Then $f(\beta) = F'a_{ijk}^{\epsilon_{l}}a_{ijl}^{\epsilon_{k}}a_{ikl}^{\epsilon_{j}}a_{jkl}^{\epsilon_{i}}B'$ and $f(\beta') = F'a_{jkl}^{\epsilon_{i}'}a_{ikl}^{\epsilon_{j}'}a_{ijl}^{\epsilon_{k}'}a_{ijk}^{\epsilon_{l}'}B'$. Assume that $i<j<k<l$. Since $n+1 \notin \{i,j,k,l \}$, the number of $a_{ik(n+1)}$ and $a_{jk(n+1)}$ before $a_{ijk}$ in $\beta$ is equal to the number of them in $\beta'$. That is, $\epsilon_{l} =  \epsilon_{l}'$. Analogously, $\epsilon_{k} =  \epsilon_{k}'$,$\epsilon_{j} =  \epsilon_{j}'$,$\epsilon_{i} =  \epsilon_{i}'$. Finally, we show that $\epsilon_{i} +\epsilon_{j} +\epsilon_{k} =0$ mod $2$. Denote $N_{ij(n+1)}$ is the number of $a_{ij(n+1)}$ in $F$. Then
$$\epsilon_{k} = N_{il(n+1)}+N_{jl(n+1)},$$
$$\epsilon_{j} = N_{il(n+1)}+N_{kl(n+1)},$$
$$\epsilon_{i} = N_{jl(n+1)}+N_{kl(n+1)}.$$
Therefore 
$$\epsilon_{k} +\epsilon_{j} +\epsilon_{i}= N_{il(n+1)}+N_{jl(n+1)}+N_{il(n+1)}+N_{kl(n+1)}+N_{jl(n+1)}+N_{kl(n+1)}=0~ mod~2.$$
Analogously we can show that the relation $a_{ijk}a_{ijl}a_{ikl}a_{jkl} = a_{jkl}a_{ikl}a_{ijl}a_{ijk}$ is preserved by $f$ for any $i,j,k,l$ in $\{1, \cdots, n\}$.

\end{proof}


Then for a word $\beta$ in $G_{n,p}^{3}$ and for $a_{ijk}$ in $\beta$, we can define $i_{a_{ijk}^{\epsilon}}^{p}(l)$ by 
\begin{center}
$i_{a_{ijk}^{\epsilon}}^{p}(l) = \left \{
\begin{array}{cc} 
    N_{ikl}^{0}+N_{jkl}^{0}+N_{ijl}^{1} ~ mod ~2 & \text{if}~\epsilon =0, l>k, \\
     N_{ikl}^{0}+N_{jkl}^{0}~ mod ~2 & \text{if}~\epsilon =0, l<k, \\
    N_{ikl}^{0}+N_{jkl}^{1}+N_{ijl}^{0}  ~ mod~ 2 &  \text{if}~\epsilon =1, l>k, \\
    N_{ikl}^{0}+N_{jkl}^{1} ~ mod~ 2 &  \text{if}~\epsilon =1, l<k . \\
   \end{array}\right. $
\end{center}
where $k>i,j$ and $N_{ikl}^{\epsilon}$ is the number of $a_{ikl}^{\epsilon}$, which appears before $a_{ijk}^{\epsilon}$.
Define a group presentation $F_{n}^{3}$ generated by $\{ \sigma ~|~ \sigma : \{1,2,\cdots n\} \backslash \{i,j,k\} \rightarrow \mathbb{Z}_{2} \}$ with relations $\{ \sigma^{2} = 1\}$. Define a word $w_{ijk}^{p}(\beta)$ in $F_{n}^{3}$ for $\beta$ by $w_{ijk}^{p}(\beta) = i_{c_{1}}i_{c_{2}} \cdots i_{c_{m}}$.

\begin{lem}
$w_{ijk}^{p}$ is an invariant for $G_{n,p}^{3}$.
\end{lem}

\begin{proof}
It suffices to show that for two $\beta$ and $\beta'$ such that $\beta'$ is obtained from $\beta$ by applying one relations of $G_{n,p}^{3}$, $w_{ijk}^{p}(\beta) = w_{ijk}^{p}(\beta')$. For relations $(a_{ijk}^{\epsilon})^{2} =1$ and $(a_{ijk}^{\epsilon_{1}}a_{lmn}^{\epsilon_{2}})^{2}=1$, it is clear. Now consider the relation $(a_{ijk}^{\epsilon_{l}}a_{ijl}^{\epsilon_{k}}a_{ikl}^{\epsilon_{j}}a_{jkl}^{\epsilon_{i}})^{2} =1$, where $t \geq i,j,k,l$ and $\sum_{s \in \{i,j,k,l \}  \backslash \{t\}}\epsilon_{s} = 0$ mod $2$. If $(a_{stu}^{\epsilon_{v}}a_{stv}^{\epsilon_{u}}a_{suv}^{\epsilon_{t}}a_{tuv}^{\epsilon_{s}})^{2} =1$, where $s<t<u<v$ and $\epsilon_{s} +\epsilon_{t} +\epsilon_{u} = 0$ mod $2$ is applied and $|\{s,t,u,v\} \cap \{i,j,k\}|<3$, then $i_{a_{ijk}}$ is preserved for every $a_{ijk}$ in $\beta$ and hence $w_{ijk}^{p}(\beta)$ is preserved. Suppose $|\{s,t,u,v\} \cap \{i,j,k\}|=3$. Then there are two $a_{ijk}^{\epsilon}$ in $(a_{stu}^{\epsilon_{v}}a_{stv}^{\epsilon_{u}}a_{suv}^{\epsilon_{t}}a_{tuv}^{\epsilon_{s}})^{2} $, say they are $c_{1}$ and $c_{2}$ in order(left to right). 

If $s=i,t=j,u=k$, then 
$$a_{ijk}^{\epsilon_{v}}a_{ijv}^{\epsilon_{k}}a_{ikv}^{\epsilon_{j}}a_{jkv}^{\epsilon_{i}}a_{ijk}^{\epsilon_{v}}a_{ijv}^{\epsilon_{k}}a_{ikv}^{\epsilon_{j}}a_{jkv}^{\epsilon_{i}} =1 $$ 
is applied. Note that $i_{c_{1}}(r) = i_{c_{2}}(r)$ for $r \in \{1,\cdots n \} \backslash \{i,j,k,v\}$, because there are no $a_{ijr}^{\epsilon}$,$a_{ikr}^{\epsilon}$ and $a_{jkr}^{\epsilon}$ between $c_{1}$ and $c_{c}$. Denote that $\alpha_{xyz}^{\epsilon}$ by the number of  $a_{xyz}^{\epsilon}$ between $c_{1}$ and $c_{2}$. Then there are two cases : $\epsilon_{v} =0$ and $\epsilon_{v}=1$. For the case $\epsilon_{v} =0$, there are 4 subcases:
\begin{enumerate}
\item $(\epsilon_{k},\epsilon_{j},\epsilon_{v}) =(0,0,0)$, 
\item $(\epsilon_{k},\epsilon_{j},\epsilon_{v}) =(0,1,1)$,
\item $(\epsilon_{k},\epsilon_{j},\epsilon_{v}) =(1,0,1)$, 
\item $(\epsilon_{k},\epsilon_{j},\epsilon_{v}) =(1,1,0)$.
\end{enumerate}

Suppose that $(\epsilon_{k},\epsilon_{j},\epsilon_{l}) =(0,0,0)$. Since there are one $a_{ikv}^{0}$, one $a_{jkv}^{0}$ and no $a_{ijv}^{1}$ and $v>k$, 
$$i_{c_{2}}(v) = i_{c_{1}}(v)+\alpha_{ikv}^{0}+\alpha_{jkv}^{0}+\alpha_{ijv}^{1}= i_{c_{1}}(v)+1+1+0 = i_{c_{1}}(v)~mod2. $$  
Suppose that $(\epsilon_{k},\epsilon_{j},\epsilon_{v}) =(0,1,1)$. Since there are no $a_{ikv}^{0}$, $a_{jkv}^{0}$ and $a_{ijv}^{1}$, 
$$i_{c_{2}}(v) = i_{c_{1}}(v) +\alpha_{ikv}^{0}+\alpha_{jkv}^{0}+\alpha_{ijv}^{1} =i_{c_{1}}(v)+0+0+0 = i_{c_{1}}(v)~mod2. $$ 
Suppose that $(\epsilon_{k},\epsilon_{j},\epsilon_{v}) =(1,0,1)$. Since  there are no $a_{ikv}^{0}$ and one $a_{jkv}^{0}$ and one $a_{ijv}^{1}$ between $c_{1}$ and $c_{2}$ and $v>k$, 
$$i_{c_{2}}(v) =  i_{c_{1}}(v) +\alpha_{ikv}^{0}+\alpha_{jkv}^{0}+\alpha_{ijv}^{1} =i_{c_{1}}(v)+0+1+1 = i_{c_{1}}(v)~mod2. $$
Analogously, we can show that $i_{c_{1}} = i_{c_{2}}$ in other cases. Since $i_{c_{1}}i_{c_{2}} = 1$, $w_{ijk}^{p}$ is preserved with respect to relations in $F_{n}^{3}$.  

If $s=i, t=j, v=k$, then 
$$a_{iju}^{\epsilon_{k}}a_{ijk}^{\epsilon_{u}}a_{iuk}^{\epsilon_{j}}a_{juk}^{\epsilon_{i}}a_{iju}^{\epsilon_{k}}a_{ijk}^{\epsilon_{u}}a_{iuk}^{\epsilon_{j}}a_{juk}^{\epsilon_{i}}=1 $$
 is applied. Note that $i_{c_{1}}(r) = i_{c_{2}}(r)$ for $r \in \{1,\cdots n \} \backslash \{i,j,k,u\}$, because there are no $a_{ijr}^{\epsilon}$,$a_{ikr}^{\epsilon}$ and $a_{jkr}^{\epsilon}$ between $c_{1}$ and $c_{c}$. Since $u<k$, by definition of $i_{c}$, $a_{iju}^{\epsilon_{k}}$ does not affect to $i_{c_{1}}(u)$ and $i_{c_{2}}(u)$. Since $\epsilon_{u} + \epsilon_{j}+\epsilon_{i} =0$ mod $2$, there are four cases:
\begin{enumerate}
\item $(\epsilon_{u},\epsilon_{j},\epsilon_{i}) =(0,0,0)$, 
\item $(\epsilon_{u},\epsilon_{j},\epsilon_{i}) =(0,1,1)$,
\item $(\epsilon_{u},\epsilon_{j},\epsilon_{i}) =(1,0,1)$, 
\item $(\epsilon_{u},\epsilon_{j},\epsilon_{i}) =(1,1,0)$.
\end{enumerate}
If $(\epsilon_{u},\epsilon_{j},\epsilon_{i}) =(0,0,0)$, then
$$ i_{c_{2}}(v) =  i_{c_{1}}(v) +\alpha_{iuk}^{0}+\alpha_{juk}^{0} =i_{c_{1}}(v)+1+1 = i_{c_{1}}(v)~mod2. $$
If $(\epsilon_{u},\epsilon_{j},\epsilon_{i}) =(1,0,1)$, then 
$$ i_{c_{2}}(v) =  i_{c_{1}}(v) +\alpha_{iuk}^{0}+\alpha_{juk}^{1} =i_{c_{1}}(v)+0+1 = i_{c_{1}}(v)~mod2. $$
Analogously, $i_{c_{1}} = i_{c_{2}}$ in other cases and the proof is completed.
It is easy to show that $w_{ijk}^{p}(\beta)$ does not change by applying $(a_{stu}^{\epsilon_{v}}a_{stv}^{\epsilon_{u}}a_{suv}^{\epsilon_{t}}a_{tuv}^{\epsilon_{s}})^{2} =1$, where $z$ is the largest index in $\{s,t,u,v\} $ and $\sum_{x \in \{s,t,u,v \} \backslash \{z\}} \epsilon_{x}$ for any $\{s,t,u,v \} \subset \{1, \cdots, n\}$ and the proof is completed.
\end{proof}

\begin{exa}
Let $$\beta = a_{124}a_{245}a_{124}a_{245}a_{234}a_{245}a_{234}a_{245}a_{245}a_{124}a_{245}a_{124}a_{245}a_{234}a_{245}a_{234}$$ in $G_{5}^{3}$. Then 
$$f(\beta) = a_{124}^{0}a_{124}^{1}a_{234}^{0}a_{234}^{1}a_{124}^{1}a_{124}^{0}a_{234}^{1}a_{234}^{0},$$
and $f(\beta)$ is in $G_{4,p}^{3}$. Let $c_{1} = a_{124}^{0}$, $c_{2} = a_{124}^{1}$, $c_{3} = a_{124}^{1}$ and $c_{4} = a_{124}^{0}$  such that the order of $c_{i}$ agrees with the order of $a_{124}^{\epsilon}$ in $f(\beta)$. For a type $(124)$ and for an index $3$, 
\begin{itemize}
\item $i_{c_{1}}(3)= N_{234}^{0}+N_{134}^{0}= 0+0 =0,$
\item $i_{c_{2}}(3)= N_{234}^{1}+N_{134}^{0}=0+0=0,$
\item $i_{c_{3}}(3)= N_{234}^{1}+N_{134}^{0}=1+0=1,$
\item $i_{c_{4}}(3)= N_{234}^{0}+N_{134}^{0}= 1+1=0.$
\end{itemize}

Hence the word, which is derived from the above calculation, is not trivial.
\end{exa}

\begin{rem}
We can obtain homomorphisms from $G_{n+1}^{3}$ to $G_{n,p}^{3}$ by omitting not only $n+1$ index, but also one of the other indices. That is, as $i_{ij}^{p}$ in section 2, we obtain invariants by adding information from a fixed index. 
\end{rem}

\end{document}